\documentclass[12pt, reqno]{amsart}
\usepackage{verbatim}
\usepackage{graphicx}
\usepackage{enumerate}
\newtheorem{theorem}{Theorem}
\newtheorem{lemma}[theorem]{Lemma}
\newtheorem{proposition}[theorem]{Proposition}

\newtheorem{conjecture}{Conjecture}
\newtheorem{question}[conjecture]{Question}
\newtheorem*{claim*}{Claim}

\theoremstyle{definition}
\newtheorem*{definition*}{Definition}

\newtheorem{remark}[theorem]{Remark}

\newtheorem{example}{Example}

\newcommand\ns[1]{ \left\{ {#1} \right\} }
\DeclareMathOperator{\metric}{d}
\DeclareMathOperator{\metricstar}{d^{*}}

\newcommand{\Z}{{\mathbb Z}}
\newcommand{\R}{{\mathbb R}}
\newcommand{\N}{{\mathbb N}}
\newcommand{\T}{{\mathsf{T}}}
\newcommand{\e}{\varepsilon}

\newcommand\X{\Omega}

\newcommand\hx{\hat{x}}

\newcommand\garbage[1]{}

\newcommand\kinitial{n_{\mathrm{initial}}}

\newcommand\kblock{{k_{\mathrm{group}}}}

\newcommand\kmark{{k_{\mathrm{mark}}}}
\newcommand\krmark{{r_{\mathrm{mark}}}}

\newcommand\miss{{{\mathrm{miss}}}}
\newcommand\rep{{{\mathrm{rep}}}}
\newcommand\type{{{\mathrm{type}}}}

\renewcommand{\P}{{\mathbb P}}
\newcommand{\nset}[1]{[#1]}

\newcommand{\dff}[1]{\textbf{\emph{#1}}}

\newcommand{\erk}{\hfill \ensuremath{\Diamond}} %%used to signal end of remark

\newcommand\ellstar{\ell^*}

\begin{document}

\title[Monotone isomorphisms]{A monotone isomorphism theorem}
\author[T. Soo]{Terry Soo}

\address{Department of Mathematics,
University of Kansas,
405 Snow Hall,
\indent 1460 Jayhawk Blvd,
Lawrence, Kansas 66045-7594}

\email{tsoo@ku.edu}
\urladdr{http://www.math.ku.edu/u/tsoo}
\thanks{Funded in part by a New Faculty General Research Fund}


\dedicatory{Dedicated to Professor Andr\'es del Junco, September 21, 1948 -- June 17, 2015}

\begin{abstract}
 In the simple case of a Bernoulli shift on two symbols, zero and one, by  permuting the symbols, it is obvious that any two equal entropy  shifts are isomorphic.  We show that the isomorphism can be realized by a factor that maps a binary sequence to another that is coordinatewise smaller than or equal to the original sequence.  
\end{abstract}

\maketitle

\section{Introduction}

Let $N$ be a positive integer, $\nset{N} = \ns{0,1, \ldots, N-1}$, and $\X = \nset{N}^{\Z}$.   Let $T: \X \to \X$ be the left-shift 
given by $(T x)_{i} = x_{i+1}$ for all $i \in \Z$.   Given a  probability measure $\mathbf{p}$   on $\nset{N}$, we call  
$B(\mathbf{p}) = (\X, \mathbf{p}^\Z, T)$ a \dff{Bernoulli-shift} on $N$ symbols.  We say that a Bernoulli shift $B(\mathbf{q})$ 
is \dff{factor} of $B(\mathbf{p})$, if there exists a measurable map $\phi: \X \to \X$ such that the push-forward of 
$\mathbf{p}^{\Z}$ under $\phi$ is $\mathbf{q}^{\Z}$ and $\phi \circ T = T \circ \phi$ on a subset of $\X$ with  
$\mathbf{p}^{\Z}$-full measure; we also call the map  $\phi$  a \dff{factor} from $B(\mathbf{p})$ to $B(\mathbf{q})$.  We say 
that the  Bernoulli shifts $B(\mathbf{p})$ and $B(\mathbf{q})$ are \dff{isomorphic} if there exists  a factor map $\phi$ from 
$B(\mathbf{p})$ to $B(\mathbf{q})$ such that its inverse $\phi^{-1}$ serves as  factor map from $B(\mathbf{q})$ to 
$B(\mathbf{p})$; in this case, we call $\phi$ an \dff{isomorphism} of $B(\mathbf{p})$ and $B(\mathbf{q})$.    A factor map 
$\phi$ is \dff{monotone} if for all $x \in \X$, we have  $\phi(x)_i \leq x_i$ for all $i \in \Z$.

\begin{theorem}
\label{mresult}
If  $p \in (\tfrac{1}{2},1)$,  then there  exists a monotone isomorphism of $B(1-p,p)$ and $B(p,1-p)$.     
\end{theorem}

Let us remark that the map defined by  $\phi(x)_i = \mathbf{1}[x_i=0]$ for all $i \in \Z$, which just swaps zeros and ones, is clearly an isomorphism of $B(1-p,p)$ and $B(p,1-p)$.  However, it is not monotone.   

It is  easy  to determine when two Bernoulli shifts are isomorphic via an invariant  introduced by Kolmogorov \cite{MR2342699}, which is non-increasing under factors and preserved under isomorphisms.
The \dff{entropy} of a probability measure $\mathbf{p}=(p_0, \ldots, p_{N-1})$ on $\nset{N}$  is given by  $H(\mathbf{p}): = -\sum_{i=0}^{N-1}{p}_i \log {p}_i$.  Sinai \cite{Sinai, MR2766434}  proved that if $H(\mathbf{p}) \geq H(\mathbf{q})$, then $B(\mathbf{q})$ is factor of $B(\mathbf{p})$, and Ornstein \cite{Ornstein, MR3052869} proved that the  entropies of two Bernoulli shifts are equal if and only if the two Bernoulli shifts are isomorphic.    

Although it is easy to compute the entropy of a Bernoulli shift and to determine whether two Bernoulli shifts are isomorphic, the actual factor map which realizes the isomorphism is in general a much more complicated object.   
In some special cases, the factor map has a simple description \cite{MR0110782, MR0143862}.   The first non-trivial example of an isomorphism is due to Melshalkin \cite{MR0110782}, which also gives a monotone isomorphism.   I thank Zemer Kosloff for  his help with the following example.    
\begin{example}[A classical example due to Melshalkin \cite{MR0110782}]
	\label{mesh}     We will adjust the treatment given in \cite{MR3220671} to ensure monotonicity.  
 Let $\mathbf{p}=(\tfrac{1}{8},\tfrac{1}{8},\tfrac{1}{8},\tfrac{1}{8}, \tfrac{1}{2})$ and $\mathbf{q} = (\tfrac{1}{4},\tfrac{1}{4},\tfrac{1}{4},\tfrac{1}{4},0)$, so that $N=5$, and $\mathbf{p}$ and $\mathbf{q}$ are probability measures on $\nset{N} = \ns{0,1,2,3,4}$.  
Let $x \in \X  =\nset{5}^{\Z}$.  We define a  factor map $\phi: \Omega \to \Omega$ such that  if $x_i=4$, then $\phi(x)_i \in \ns{2,3}$,  if $x_i \in \ns{2,3}$, then $\phi(x)_i = 1$, and if $x_i \in \ns{0,1}$, then $\phi(x)_i = 0$.    
%
%
\begin{comment}
   It will be convenient to have another set of names for $[N]$.  When referring to the measure $\mathbf{p}$, we say  $3=$  `black and red', $2=$ `black and blue', $1=$  `white and red', and finally $0=$  `white and blue.'   When referring to measure $\mathbf{q}$, we say $3 =$`rouge', $2=$ `bleu', $1=$ `noir', and $0=$ `blanc'.       
\end{comment} 
%
%

It remains to specify what happens when $x_i=4$.  
  Think of every $x_i=4$ as a right parenthesis, and think of every $x_i \not =4$ as a left parenthesis.    Ergodicity implies that every parenthesis will be matched legally  almost surely.  If  $x_i=4$, then let  $j$ be the position of the corresponding left parenthesis.   If $x_j$ is odd, then we set $\phi(x)_i =3$,  if $x_j$ is even, then we set $\phi(x)_i =2$. 
\begin{comment}
If $x_i=4$, then set $\phi(x)_i$ to be the value of the corresponding second color, in French,  of the left parenthesis; otherwise, set $\phi(x)_i$ to be the value of the corresponding first color in French.  
  \end{comment}
%

      By definition, the map $\phi$ satisfies $\phi \circ T = T \circ \phi$ and is monotone.    Melshalkin proved that $\phi$ is an isomorphism of $B(\mathbf{p})$ and $B(\mathbf{q})$. \erk
\end{example}

It is easy to see that a necessary condition for the existence of a monotone factor from $B(\mathbf{p})$ to $B(\mathbf{q})$ is that there exists a \dff{monotone coupling} of $\mathbf{p}$ and $\mathbf{q}$; that is, a probability measure $\mathbf{\rho}$ on $\nset{N} \times \nset{N}$ such that $\mathbf{\rho}(\cdot, \nset{N}) = \mathbf{p}$,  $\mathbf{\rho}(\nset{N}, \cdot) = \mathbf{q}$, and $\mathbf{\rho}\ns{(n,m):  n\geq m}=1$.   By  Strassen's theorem \cite{Strassen}, the existence of a monotone coupling is equivalent to the condition that $\sum_{i=0}^k{p}_i \leq  \sum_{i=0}^k {q}_i$ for all
$0 \leq k < N$, in which case we say that   $\mathbf{p}$ \dff{stochastically dominates} $\mathbf{q}$.

\begin{theorem}[Quas and Soo \cite{Qsc}]
\label{old}
Let $\mathbf{p}$ and $\mathbf{q}$ be probability measures on $\nset{N}$.  If $\mathbf{p}$ stochastically dominates $\mathbf{q}$ and $H(\mathbf{p})$ is strictly greater than $H(\mathbf{q})$, then there exists a monotone factor from $B(\mathbf{p})$ to $B(\mathbf{q})$.  
\end{theorem}

 Karen Ball \cite{Ball} proved Theorem \ref{old} in the case that the measure $\mathbf{q}$ is supported on two symbols.  In both of  those papers,  a strict entropy inequality is required.  In this paper, we treat the case of equal entropy,  in the special case where there are only two symbols, zero and one,  in each of the Bernoulli shifts.    The methods used to prove  Theorem \ref{mresult} can also be adapted to produce monotone factors in other  specific cases, but we do not know the answer to the following question.

\begin{question}  
Let $\mathbf{p}$ and $\mathbf{q}$ be probability measures on $\nset{N}$ such that $\mathbf{p}$ stochastically dominates $\mathbf{q}$ and $H(\mathbf{p}) = H(\mathbf{q})$.   Does there exists a monotone factor from $B(\mathbf{p})$ to $B(\mathbf{q})$?
\end{question}

Russell Lyons \cite{Ball}  first posed the   question of whether a monotone factor exists between two Bernoulli shifts.   The requirement of monotonicity makes defining  maps more difficult.  In a  related problem, Gurel-Gurevich and Peled \cite[Theorem 1.3]{GGP} proved that for $p \in (\tfrac{1}{2}, 1)$ there exists a  monotone map $\phi:\ns{0,1}^{\Z} \to \ns{0,1}^{\Z}$ such that the product measure $(p, 1-p)^{\Z}$ is the push-forward of $(1-p, p)^{\Z}$ under $\phi$; however,  their map is  \emph{not} be equivariant; that is, it does not satisfy $\phi\circ T = T \circ \phi$.      

See \cite{MR0304616, MR3052869, MR2342699} for more information on entropy and the isomorphism problem in ergodic theory. 
See \cite{Qsc} and \cite{iidtrees}  for background on   factors in probability theory.

The proof of Theorem \ref{mresult} will involve some of the methods of  \cite{Qsc}, which in turn combines ideas from various treatments of the Ornstein and Sinai factor theorems  given by  Keane and Smorodinsky \cite{keanea, keaneb},   Burton  and Rothstein \cite{BurKeaSer,BurRot}, del Junco \cite{Juncoa, Junco}, and Ball \cite{Ball}.   We  briefly summarize some of the main  features and differences  in their proofs.  Keane and Smorodinsky, and Ball employed a marker-filler method and a version of  Hall's marriage theorem (see Remark \ref{KSdel}).   Del Junco  also employed a marker-filler method, but he replaced the marriage lemma with his star-coupling (see Section \ref{couplings}).   These constructions are explicit and they exhibit factor maps that are finitary--an almost surely continuity property (see \cite{MR2306207} for details).      In a somewhat more abstract approach, Burton and Rothstein proved  that in a suitably defined metric space, the set of all  factors is a residual set, in the sense of the Baire category theorem.  This was the approach taken in \cite{Qsc}, and will also be the approach we take here.  

\section*{Dedication}

I never had the pleasure of meeting Professor del Junco, but I wrote to him in December 2013 about Theorem \ref{old} with a preprint of \cite{Qsc}.  He wrote back the same day saying he was glad that an old idea of  his had found another application and that he  always felt that the star-coupling  was one of his best ideas.  

His coupling was a key  feature in our proof of Theorem \ref{old}, and will also be a star feature in the proof of Theorem \ref{mresult}.

\section{Coupling and Stochastic domination}

Strassen's theorem \cite{Strassen} holds in the much more general setting of a partially ordered  Polish space.  The proof,  even in the case of a finite set is non-trivial, see for example \cite[Theorem 10.4]{LP:book}.  However, in the special case of real-valued random variables or random variables taking values on a finite totally ordered set,
 the proof is easily obtained using a simple coupling of random variables.  

\subsection{Quantiles}

Let $X$ be a real-valued random variable, with cumulative distribution function or law given by $F(z)=F_X(z):= \P(X \leq z)$ for all $z \in \R$.  Define the generalized inverse of $F$ via $F^{-1}(y) := \sup\ns{x \in \R: F(x) <y}$.  Let $U$ be  uniformly distributed in $[0,1]$, so that 
$F_U(z) = z$ 
for all $z \in [0,1]$.   We call $F_X^{-1}(U)$ the \dff{quantile representation} of $X$. It is easy to see that   the random variable $F_X^{-1}(U)$ has the same law  as $X$.  When we define random variables using the quantile representation sometimes we will refer to the random variable $U$ as the  \dff{randomization}; often $U$ will be chosen to be independent of any previously defined random variables.      

If  $X$ and $Y$ are two real-valued random variables, we say that $X$ \dff{stochastically dominates} $Y$ if $\P(X \leq z) \leq \P(Y \leq z)$ for all $z \in \R$.  A \dff{coupling} of $X$ and $Y$ is a pair of random variables $(X', Y')$ defined on the same probability space such that $X'$ has the same law as $X$ and $Y'$ has the same law as $Y$.   Let $U$ be uniformly distributed in $[0,1]$.  If we set $X':= F_X^{-1}(U)$ and $Y':=F_Y^{-1}(U)$,  then the \dff{quantile coupling} of $X$ and $Y$ is given by $(X', Y')$.       We say that the coupling $(X',Y')$ is \dff{monotone} if $X' \geq Y'$.  Strassen's theorem implies that $X$ stochastically dominates $Y$ if and only if there exists a monotone coupling of $X$ and $Y$.    Clearly, the existence of a monotone coupling implies stochastic domination; on the other hand, it is easy to see that the quantile coupling is monotone under the assumption of stochastic domination.

Let us remark that stochastic domination and the quantile coupling are also similarly defined in the case that the random variables take values in a finite totally ordered space.

\begin{lemma}[Strassen's theorem via the quantile coupling]
\label{easyS}
Let $X$ and $Y$ be real-valued random variables or random variables taking values in a finite totally ordered space.  If $(X',Y')$ is a quantile coupling of $X$ and $Y$, then 
$X'$ is almost surely greater than or equal to $Y'$ 
if and only if $X$ stochastically dominates $Y$.  
\end{lemma}

In Section \ref{couplings}, we will discuss an ingenious variation of the quantile coupling due to del Junco \cite[Section 4]{Juncoa}, which will be a key ingredient in our proof of Theorem \ref{mresult}.

\subsection{An simple application of Strassen's theorem}

Lemma \ref{easyS} will be used to prove the following simple observation,  which will serve as the starting point in our proof of  Theorem \ref{mresult}. For two binary sequences $x$ and $y$  of the same length, we write $x \preceq y$ if and only if $x_i \leq y_i$ for all indices $i$.  Thus the relation $\preceq$ defines a partial order on the set of binary sequences with the same length.  We write $x=1^{n}0^{\ell}$ to mean a binary sequence of $n$ ones followed by $\ell$ zeros.  

\begin{lemma}
\label{sd}
Let $n \geq 1$ and $p \in (\tfrac{1}{2}, 1)$.  Let $X=(X_1, \ldots, X_n)$ and $Y= (Y_1, \ldots, Y_n)$ be an i.i.d.\ sequences of Bernoulli random variables with parameters $p$ and $1-p$, respectively.    Let $B_n$ be the set of size $n+1$ of all binary sequences $z$ of length $n$ of the form $z=1^{n-\ell}0^{\ell}$ for some $\ell \in [0, n]$.  Let  $X^{*}$ and $Y^{*}$ be random variables that have laws  $X$ and $Y$ conditioned to be in $B_n$, respectively.  Then with respect to the order $\preceq$, defined on binary sequences,  $X^{*}$ stochastically dominates $Y^{*}$, and there is a monotone coupling of $X^{*}$ and $Y^{*}$.
\end{lemma}

Note that in  Lemma \ref{sd}, although the set of all binary sequences of a fixed length is only partially ordered by $\preceq$,  the set $B_n$ is totally ordered by $\preceq$.  The set $B_n$ can also be described as  the set of binary sequences of length $n$ that do not have a zero followed by a one.  We will refer to $B_n$ as a \dff{filler} set.

\begin{proof}[Proof of Lemma \ref{sd}]   
Lemma \ref{sd}	is simple consequence of the duality between $p$ and $1-p$. 
For every integer $\ell \in [0,n]$, we have
\begin{equation}
\label{dual}
\sum_{i=0}^{n- \ell} p^{n-\ell-i}(1-p)^i = \sum_{i=0}^{n- \ell} (1-p)^{n-\ell-i}p^i;
\end{equation}
 this implies, using $\ell=0$,  that  
$\P(X \in B_n) = \P(Y \in B_n)$.
 Thus in order to prove that $X^*$ stochastically dominates $Y^*$, it  suffices to show that for all $z \in B_n$, we have $\P(X \preceq z, X \in B_n) \leq \P(Y \preceq z, Y \in B_n).$
If $z =1^{n-\ell}0^{\ell}$, then since $p > 1-p$, by equality \eqref{dual} we have  
\begin{eqnarray*}
\P(X \preceq z, X \in B_n) &=&(1-p)^{\ell}\Big(\sum_{i=0}^{n- \ell} p^{n-\ell-i}(1-p)^i\Big) \\
&\leq& p^{\ell}\Big(\sum_{i=0}^{n- \ell} (1-p)^{n-\ell-i}p^i\Big) \\
&=&  \P(Y \preceq z, Y \in B_n).
\end{eqnarray*}

The existence of a monotone coupling follows from Lemma \ref{easyS}.
\end{proof}

\begin{comment}
\begin{remark}
	\label{ooo}
The fact that $B_n$ is linearly ordered makes the proof of Lemma \ref{sd} simple.  In Lemma \ref{sd}, if we replace $B_n$, the set of all binary sequences of length $n$ that do not contain an instance of $01$, with the set of all binary sequences that do not contain an instance of $0011$, we suspect Lemma \ref{sd} still holds, but do not know how to prove it.     \erk
\end{remark}
\end{comment}

\section{Markers, fillers, and joinings}

\subsection{Markers}

Let us fix $\X = \ns{0,1}^{\Z}$.  Let $x \in \X$.   We call the interval  $[i, i+1] \subset \Z$ a \dff{primary marker} if $x_i=0$ and $x_{i+1}=1$.   Later, we will define secondary and tertiary markers which will consist of consecutive primary markers.     Note that two distinct primary markers have an empty intersection.      We call an interval of $\Z$ a \dff{filler} if it is nonempty and lies between two primary markers.  Thus each $x \in \X$ partitions $\Z$ into intervals of primary markers and fillers.   

Let $p \in (\tfrac{1}{2},1)$, and consider the product probability measures on $\X=\ns{0,1}^{\Z}$ given by $\mu:=(1-p, p)^{\Z}$ and $\nu:= (p, 1-p)^{\Z}$.  Thus the probability that the zeroth coordinate is a one under $\mu$ is $p$ and is  $1-p$ under $\nu$. 
   By conditioning,  an instance of a random variable $X$ with law $\mu$  can be given by first deciding on the locations of the primary markers, and then deciding on the content of the filler; the same observation holds for a random variable $Y$ with law $\nu$.

   To be more precise, let $\T=\ns{\mathsf{M},\mathsf{F}}^{\Z}$, where $\mathsf{M}$ and $\mathsf{F}$ are two symbols that stand for `marker' and `filler.'  For each $x \in \X$, define  the \dff{hat map} by setting
  \begin{eqnarray*}
  	\hx(i) = \begin{cases} \mathsf{M} \text{ if }  i\in \Z \text{ is in a primary marker;} &  \\
  		\mathsf{F} \text{ otherwise}. 
  	\end{cases}
  \end{eqnarray*}
Let $\tau$ and $\tau'$ be push-forwards of the  measures $\mu$ and $\nu$ via the hat map.  Sometimes we will refer to $\tau$ as the \dff{marker measure}.       We have the  following disintegration.  For $\tau$-almost every $t\in \T$, there exists  a probability measure,  $\mu_{t}$ on  $\X$, such that  
  $$ \int f(x) d\mu(x) = \int \Bigg(\int  f(x) d\mu_{t}(x) \Bigg)d\tau(t)$$
for all measurable $f:\X \to [0, \infty)$.  

\begin{remark}
	\label{iid}
Keane and Smorodinsky \cite[Lemma 4]{keanea} give a concrete description of $\mu_t$.    The measure $\mu_{t}$ assigns the sequence $01$ to each primary marker interval of $t$, and  is a product measure on the filler intervals, where on a filler interval of length $n$ it is the law of $n$ i.i.d.\ Bernoulli random variables with parameter $p$ conditioned to be in the set $B_n$ of sequences of consecutive ones followed by consecutive zeros (see Lemma \ref{sd}).  The analogous result holds of $\nu$.  \erk
\end{remark}

\begin{remark}
	\label{samemark}
  Notice that the probability that the origin is contained in a primary marker is same under $\mu$ and $\nu$.   Keane and Smorodinsky \cite[Lemma 3]{keanea} proved that $\tau = \tau'$.   Thus the marker measure $\tau$ is the same for  $\mu$ and $\nu$ and depends only on the parameter $p$.   This  fact will also be important in our proof of Theorem \ref{mresult}. \erk
\end{remark}  
  
  \subsection{Joinings}
   A \dff{coupling} of $\mu$ and $\nu$ is a probability measure $\xi$  on $\X \times \X$ that has marginals $\mu$ and $\nu$; a \dff{joining} is a coupling  that is invariant under the product shift $T \times T$, so that $\xi \circ (T \times T) = \xi$.  A joining $\xi$ is \dff{ergodic} if all $\xi$-almost sure $(T \times T)$-invariant sets have measure zero or one.      A coupling $\xi$ is \dff{monotone} if 
   $$\xi\ns{(x, y)\in \X \times \X:  x_i \geq y_i \text{ for all } i \in \Z }=1.$$

       A joining $\xi$ is of \dff{marker form} if for $\xi$-almost every $(x,y)\in \X \times \X$ the binary sequences $x$ and $y$ have the same primary markers.  It   follows from Remark \ref{samemark} that there exists a joining of $\mu$ and $\nu$ in marker form.  We will use a monotone version of this fact.
 
 \begin{proposition}
\label{KSmark}
 There exists a  monotone joining of $\mu$ and $\nu$ of marker form.
 \end{proposition}  

\begin{proof}
By Remark \ref{samemark}, we have  $\tau = \tau'$.  Hence we may assume that there exist
 random variables $X$ and $Y$ with laws $\mu$ and $\nu$ such that $X$ and $Y$ have the same primary markers and filler intervals.   Consider a coupling of $X$ and $Y$ defined in the following way.  By Remark \ref{iid}, conditioned on the locations of the primary markers, for each filler interval $I$ of $X$, we know that the law of  restrictions of $X$ to $I$ is given by the law of  a finite sequence of i.i.d.\ Bernoulli random variables with parameter $p$ conditioned be in a filler set; furthermore, conditioned on the locations of primary markers,  the restrictions of $X$ to each filler interval give  independent random variables.  The analogous statement holds for $Y$.   For each filler interval $I$, by Lemma \ref{sd}, there exists a monotone coupling of the restriction of $X$ to $I$ and the restriction of $Y$ to $I$.  Hence by applying Lemma \ref{sd} to each of the filler intervals independently, and leaving the primary  markers alone,  we obtain a coupling $(X', Y')$ of $X$ and $Y$ whose law is monotone and of marker form.    
\end{proof}

\subsection{The Baire category approach of Burton and Rothstein}
\label{baire}

Let $p \in (\tfrac{1}{2}, 1)$ and  $J=J(p)$ be the set of all monotone ergodic joinings of $\mu=(1-p, p)^{\Z}$ and $\nu=(p, 1-p)^{\Z}$ of marker form.        Note that $J$ is nonempty by Proposition \ref{KSmark}.  Following the approach of   Burton and Rothstein \cite{BurRot}, we will show the  
monotone isomorphisms 
are a  residual set in $J$, when we endow $J$ with a suitable topology.    Following del Junco  \cite{Junco}, we assign a complete metric to $J$ as follows.     
 For $i\geq 0$, let $\mathcal C_i$ be the set of  measurable
$C \subset  \X \times \X$ that only depend on the coordinates
$j\in [-i,i]$; we will call such sets \dff{cylinder sets}.  
For any two measures $\zeta$ and $\xi$ on $\X \times \X$ (which may not be joinings), set
\[
\metricstar (\zeta, \xi):=  \sum_{i=0}^ \infty 2^{-(i+1)}
\sup_{C \in \mathcal C_i} |\zeta(C) - \xi(C)|.
\]
Thus $\metricstar$ is the usual weak-star metric.  
For $\xi \in J$, let $\xi_{t}$ be $\xi$ conditioned to have the primary  markers given by  $t \in \T$.    Let us remark that $\xi_t$ is no longer a joining.   Let $\tau$ be the common marker measure.   For $\zeta, \xi \in J$,     set
\begin{equation}
\label{metric}
 \metric(\xi, \zeta):= \int \metricstar(\xi_{t}, \zeta_{t}) d\tau(t).
\end{equation}

Standard methods show that $(J, \metric)$ is a Baire space   (see for example \cite[Lemma 17]{Qsc}).  We will show that the set of monotone isomorphisms contains an intersection of open dense sets of $J$, and hence is nonempty by the Baire category theorem.  
To be more precise, 
let $\mathcal F$ denote the product sigma-algebra for  $\X$.
Let
$\mathcal P := \ns{P_0, P_1}$ 
denote the partition of  $\X$ according the zeroth coordinate
so that $P_i := \ns{x \in \X: x_0 = i}$. Let  $\zeta \in J$ 
and let $\e > 0$.    If there exists $\T' \subset \T$ with $\tau(\T') >1-\e$ such that for every $t \in \T'$,  and each  $P \in \sigma(\mathcal P)$ there exists a
$P' \in \mathcal F$ such that
\[
\zeta_t\Big( (P' \times  \X)  \ \triangle
\  (\X \times P) \Big) < \e,
\]
then we say that $\zeta$ is an
\dff{$\e$-almost factor} from $B(1-p, p)$  to $B(p, 1-p)$. For each $\e>0$, let $U_{\e}$ be the set of all $\e$-almost factors from $B(1-p,p)$ to $B(p, 1-p)$.    It is routine to verify that $U_{\e}$ is an open set (see for example \cite[page 123--24]{down}) and that an element in the intersection of all the $U_{\e}$ defines a  monotone factor from $B(1-p,p)$ to $B(p,1-p )$ (see for example \cite[Theorem 2.8]{rue}).    The real  work lies in verifying that $U_{\e}$ is dense; once this has been proved, the Baire category theorem gives that the set of monotone factors from $B(1-p, p)$ to $B(p, 1-p)$ contains an intersection of open dense sets, and hence is nonempty.

Theorem \ref{mresult} asserts the existence of a monotone isomorphism which appears to be a much stronger statement the existence of a monotone factor.    However, one of the advantages of the Baire category approach is that proving the existence of the isomorphism requires little  additional work.  We define an approximate factor from $B(p, 1-p)$ to $B(1-p, p)$ in the analogous way.    Let  $\zeta \in J$ 
and let $\e > 0$.    If there exists $\T' \subset \T$ with $\tau(\T') >1-\e$ such that for every $t \in \T'$,  and each  $P \in \sigma(\mathcal P)$ there exists a
$P' \in \mathcal F$ such that
$\zeta_t\big( (P \times  \X)  \ \triangle
\  (\X \times P') \big) < \e,$
then we say that $\zeta$ is an
\dff{$\e$-almost factor} from $B(p, 1-p)$ to $B(1-p, p)$.  For each $\e>0$, let $V_{\e}$ be the set of all $\e$-almost factors.   Again, one can  verify that $V_{\e}$ is an open set, and that an element in the intersection of all the $V_{\e}$  defines a monotone  factor from $B(p, 1-p)$ to $B(1-p, p)$.   Moreover,  any element in the grand intersection of all the $U_{\e}$ and $V_{\e}$ defines a monotone isomorphism.    It will become apparent that the same proof that shows that $U_{\e}$ is dense can be essentially  copied to show that $V_{\e}$ is dense.  Thus the Baire category theorem shows that the grand intersection is nonempty.

It remains to verify that for each $\e>0$, the set $U_{\e}$ of    $\e$-almost factors is dense.  Given  $\e >0$ and $\xi \in J$, we need to find  $\xi' \in U_{\e}$ with $\metric(\xi, \xi')<\e$.  We will define $\xi'$ as a certain perturbation of $\xi$ which will be obtained using del Junco's star-coupling \cite[Section 4 and Proposition 4.7]{Juncoa}.

\section{The star-coupling} 
\label{couplings}

Let $X$ and $Y$ be random variables taking values on finite sets $A$ and $B$, respectively.  In this section, we will discuss various couplings of $X$ and $Y$; that is, random variables $X'$ and $Y'$ defined together on the same probability space with the same distribution as $X$ and $Y$, respectively.

Let $\rho$ be a joint probability mass function for $X$ and $Y$.   
%Let $B' \subset B$.  
We say that an element $a \in A$ is \dff{split} by $\rho$
%in $B'$ 
if  there exist distinct $b,b' \in B$, such that $\rho(a,b)>0$ and  $\rho(a,b')>0$.  For the purposes of defining factors, we are interested in couplings that do not split many elements.  

\begin{remark}
	\label{usefula}
   Let us remark that if we assign an arbitrary total ordering to $A$ and $B$, then the law of a quantile coupling of $X$ and $Y$ will split at most $|B|-1$ elements of $A$. \erk
\end{remark}   
  
\begin{remark}
\label{KSdel}
   Keane and Smorodinsky \cite[Theorem 11]{keanea} proved that there is a coupling of $X$ and $Y$ with law $\rho'$ that will split at most $|B|-1$ elements of $A$ and in addition, $\rho'$ is {absolutely continuous} with respect to $\rho$; that is, $\rho(a,b) = 0$ implies $\rho'(a,b)=0$.  A version of their theorem  was used in the proof of Theorem \ref{old}, but we will not need to appeal to this result in our proof of Theorem \ref{mresult}. \erk
\end{remark}

\begin{comment}
\begin{proposition}[Keane and Smorodinksy's marriage theorem]
\label{marr}	
Let $X$ and $Y$ be random variables taking values on finite sets $A$ and $B$ with a  joint probability mass function $\rho$.  For all $B' \subset B$, there exists a coupling of $X$ and $Y$ with law $\rho'$, such that $\rho'$ splits in $B'$ at most $|B'|-1$ elements of $A$ and $\rho'$ is absolutely continuous respect to $\rho$.        
\end{proposition}
\end{comment}

Let $X$ and $Y$ be  jointly distributed random variables taking values on totally ordered finite sets $(A, \leq)$ and $(B , \leq)$, respectively.   Let $X'$ have the same law as $X$. One way to generate another random variable $Y'$ so that $(X', Y')$ has the same joint distribution as $(X,Y)$ is to  appeal to a quantile representation.  Consider the set of conditional cumulative distribution functions given by $Q_a:=\P(Y \leq b \ | \ X=a)$ for each $a \in A$.  Let  $U$ be uniformly distributed in $[0,1]$ and independent of $X'$.  Set $Y' :=Q_{X'}^{-1}(U)$.  It is easy to verify that $(X', Y')$ has the same joint distribution as $(X, Y)$; we call $(X',Y')$ the \dff{conditional quantile representation} of $(X,Y)$.

The next coupling we discuss is due to del Junco \cite{Juncoa, Junco}.      
 Let $(X_1, Y_1)$ and $(X_2, Y_2)$ be random variables taking values on the finite sets $(A_1,B_1)$ and $(A_2, B_2)$, respectively. 
 Suppose that each of the  sets $A_1, A_2, B_1,$ and $B_2$ are totally ordered sets.   
We will define $(X_1', Y_1')$ and $(X_2', Y_2')$ such that  $(X_i', Y_i')$ has the same law as $(X_i, Y_i)$ for $i=1,2$.       Let $U_1, U_2$, and $U$ be independent random variables uniformly distributed in the unit interval $[0,1]$.  Let $X_2'$ and $Y_1'$ be independent random variables that have the same laws as $X_2$ and $Y_1$, respectively; more specifically, we may assume that they are given by their respective quantile representations with sources of randomization given by $U_2$ and $U_1$.  Next, using the {\em same} source of randomization $U$,  let $Y_2'$ be such that $(X_2', Y_2')$ is the conditional quantile representation  of $(X_2, Y_2)$,  and  let $X_1'$ be such that $(Y_1', X_1')$ is the conditional quantile representation of  $(Y_1, X_1)$.    We refer to  $((X_1', Y_1'), (X'_2, Y'_2))$ as the  \dff{star-coupling} of $(X_1, Y_1)$ and $(X_2, Y_2)$.   

\begin{remark}   It is immediate from the definition the star-coupling that $X_2'$ is independent of $(Y_1', X_1')$ and $Y_1'$ is independent of $(X_2', Y_2')$.   \erk

\end{remark}

\begin{comment}
If we wanted to generate an instance of the star coupling on a computer, then one could proceed as follows.    Let $U_1, U_2$, and $U$ be independent random variables uniformly distributed in the unit interval $[0,1]$.  Generate $Y_1$ using $U_1$, so that $Y_1 = g_1(U_1)$ for some deterministic function $g_1$.   Similarly,  let $X_2=f_2(U_2)$ for a deterministic function $f_2$.  Suppose that $Y_1=y_1$ and $X_2=x_2$.  Let $f_{y_1}$ be the generalized inverse of the cumulative distribution function for the random variable that has law $X_1$ conditioned on the event that $Y_1=y_1$.  Similarly, let $g_{x_2}$ be the  generalized inverse of the cumulative distribution function for the random variable that has law $Y_2$ conditioned on the event that $X_2=x_2$.   Set $X_1=f_{y_1}(U)$ and $Y_2=g_{x_2}(U)$.  Note that in this description of the star coupling, it is immediate that $X_1$ is independent of $X_2$, since they are functions of independent random variables.    Similarly, $Y_1$ is independent of $Y_2$.  
\end{comment}

\begin{remark}
	\label{usefulb}
	It follows from Remark \ref{usefula}, that the star-coupling of the
	random variables $(X_1, Y_1)$ and $(X_2, Y_2)$ taking values on
	$(A_1, B_1)$ and $(A_2, B_2)$, respectively,  has the property that
	for a fixed $a_2 \in A_2$ and $b_1 \in B_1$,  the number of  $a_1 \in A_1$
	such that there are distinct $b_2, b_2' \in B_2$ with both $(a_1,b_1, a_2,b_2)$
	and  $(a_1,b_1, a_2,b_2')$ receiving positive mass under the law of  star-coupling
	$(X_1', Y_1', X_2', Y_2')$ is at most $|B_2| -1$. \erk
\end{remark}

\begin{remark}
del Junco  refers to his coupling as the $*$-joining \cite{Juncoa, Junco}.   \erk
\end{remark}

 We may also iterate the star-coupling to more than two pairs of random variables.  For example, if $(X_i, Y_i)$ are finite-valued random variables taking values in totally ordered spaces $(A_i, B_i)$ for $i=1,2,3$, we define its  \dff{iterated star-coupling} in the following way.  
Take the star-coupling of $(X_1, Y_1)$ and $(X_2, Y_2)$ to be given by  $((X_1', Y_1'), (X_2', Y_2'))$.  Assign a lexicographic ordering to the set $A_1 \times A_2$,  and take the star-coupling of $((X_1', X_2'), (Y_1', Y_2'))$ and $(X_3, Y_3)$, to obtain random variables $((X_1^{\prime \prime}, X_2^{\prime\prime}, X_3^{\prime \prime}), (Y_1^{\prime \prime}, Y_2^{\prime \prime}, Y_3^{\prime \prime}))$.   
Notice that by definition of the star-coupling,   $(X_i^{\prime \prime}, Y_i^{\prime \prime})$  has the same law as $(X_i, Y_i)$ for $i=1,2,3$.  

We will make use of the following variation of the  iterated star-coupling.   Let $\rho$ be a probability measure on the finite set $A \times B$ which has projections $\alpha$ and $\beta$ on the sets $A$ and $B$, respectively.     
For every $k \in \Z^{+}$, let $\alpha^k$  and $\beta^k$ denote the $k$-fold product measures on  $A^k$ and $B^k$, respectively.  
 Let $\kinitial$ and $\kblock \geq 2$ be integers.  Let $Z_i=(X_i, Y_i)$ be random variables with the following \dff{grouping property} with law $\rho$ and constants $\kinitial$ and $\kblock$.       
For each $i \geq 1$, the random variable $Z_i$ takes values on $(A \times B)^{\kblock} \equiv A^{\kblock} \times B^{\kblock}$ and has a law that has projections  $\alpha^{\kblock}$ and $\beta^{\kblock}$ on $A^{\kblock}$ and $B^{\kblock}$, respectively, and a projection $\rho$  on each copy of $A \times B$.         
Similarly, for $i=0$, the random variable $Z_0$ takes values on $(A \times B)^{\kinitial}$ and has a law that has projections $\alpha^{\kinitial}$ and $\beta^{\kinitial}$ on $A^{\kinitial}$ and $B^{\kinitial}$, respectively, and a projection $\rho$  on each copy of $A \times B$.   

 For $i \geq 1$,   write $X_i = (X_i^1, \ldots, X_i^{\kblock})$, $Y_i = (Y_i^1, \ldots, Y_i^{\kblock})$, $Z_i^j = (X_i^j, Y_i^j)$,    $Y_i^{\miss} = (Y_i^1, \ldots, Y_i^{\kblock - 1})$, and  $Z_i^{\miss} = (X_i, Y_i^{\miss})$.          

First, consider the following coupling of $X_0, X_1, Y_0$, and  $Y_1$.   The resulting coupling will {\em not} be a coupling of $Z_0=(X_0,Y_0)$ and $Z_1=(X_1, Y_1)$, but the resulting coupling as a measure on $(A \times B)^{\kinitial + \kblock}$  will have $\rho$ as a projection on each copy of $A \times B$.    
 Let $W =(E, F)$ have law given by the product measure  $\rho^{\kblock}$.  
Let $(\underline{Z_0}, \underline{Z_1}^{\miss})$ be a star-coupling of $Z_0=(X_0, Y_0)$ and $W^{\miss} = (E, F^{\miss})$.    
Using independent randomization, let $Y^{\rep}_1$ 
be such that the pair $(\underline{X_1}^{\kblock}, Y_1^{\rep})$ is  the conditional quantile representation of    
$W^{\kblock}=(E^{\kblock}, F^{\kblock})$. It is easy to verify from the properties of the star-coupling and the independence of $(W^1, \ldots, W^{\kblock})$  that  $Y_1^{\rep}$ is independent of  $\underline{Y_1}^{\miss}$; moreover,  $\big(\underline{Z_0}, (\underline{X_1}, (\underline{Y_1}^{\miss}, Y_1^{\rep})     \big)$ is a coupling of $Z_0$ and $W$ such that $(\underline{X_0}, \underline{X_1})$ has law $\alpha^{\kinitial + \kblock}$ and  $(\underline{Y_0}, \underline{Y_1}^{\miss}, Y_1^{\rep})$ has law $\beta^{\kinitial + \kblock}$. We will refer to this coupling as the \dff{star-coupling with replacement} of $Z_0$ and $Z_1$.    

Here, two `replacements' take place, $Z_1$ was replaced by $W=(E, F)$ which has the product measure $\rho^{\kblock}$ as its law, and we only applied the star-coupling to  $Z_0$ and $W^{\miss}$, where in the final construction, the  `missing' value is replaced with a  conditional quantile representation.

We iterate this construction as follows.  First, let $(Z_0', Z_1')$ be the  star-coupling with replacement of $Z_0$ and $Z_1$.      Next, we take the star-coupling with replacement of $\big( (X_0', X_1'), (Y_0', Y_1') \big)$ and $Z_2$; to obtain random variables $\big( (\underline{X_0'}, \underline{X_1'}, \underline{X_2}), (\underline{Y_0'}, \underline{Y_1'}, \underline{Y_2}^{\miss}, Y_2^{\rep})\big)$ taking values on 
the space 
$A^{\kinitial +2\kblock} \times B^{\kinitial +2 \kblock}$ with a law that has projections $\alpha^{\kinitial+2\kblock}$ and $\beta^{\kinitial +2\kblock}$, respectively.   Finally, it is clear that this construction can be extended an arbitrary number of times in the obvious way.    We call this construction the \dff{iterated-star coupling with replacement} of $Z_0, Z_1, \ldots Z_n$.

The importance of the star-coupling can be summarized in Proposition \ref{delJunco}, below; it is a version of del Junco's \cite[Proposition 4.7]{Junco}.     
\begin{comment}
For a random variable $X$  taking values on a finite set $A$, with a probability mass function $\alpha$, recall that the entropy of $X$ is given by
$$ H(X):= H(\alpha)= -\sum_{a \in A} \alpha(a) \log\alpha(a).$$
\end{comment}

\newpage

\begin{proposition}[del Junco]
	\label{delJunco}
	
Let $\rho$ be a probability measure on the finite set $A \times B$ and have marginals $\alpha$ and $\beta$, on $A$ and $B$, respectively.     Assume that $H(\alpha) = H(\beta)$.     Let $\kblock \geq 2$.
    For $\eta >0$, there exists   $\kinitial = \kinitial(\eta, \kblock) \in \Z^{+}$
	such that the following holds.    
	
 Let $n \in \Z^{+}$.   Let $Z_i=(X_i, Y_i)$, for $i=0,1, \ldots, n$, have the grouping property with the law $\rho$ and constants $\kinitial$ and $\kblock$.    
	Define the following product spaces
	$$\mathbf{I}_j := A^ {\kinitial} \times A^{\kblock j} \equiv A^{\kinitial + \kblock j},$$
%	and 
		$$\mathbf{J}_j := B^ {\kinitial} \times B^{\kblock j} \equiv B^{\kinitial + \kblock j},$$ 
and
		$$\mathbf{\bar{J}}_j := B^{(\kblock-1)j}.$$
 For $\mathbf{y} = (y_0,  (y_1, y_1'), \ldots,  (y_j, y_j')) \in   {\mathbf{J}_j} = B^{\kinitial} \times B^{\kblock j} = B^{\kinitial} \times (B^{\kblock -1} \times B) \times  \cdots \times (B^{\kblock -1} \times B)$, let  $$\mathbf{\bar{y}} = (y_1, \ldots, y_j) \in \mathbf{\bar{J}}_j.$$ 
	Let $\mathbf{W}_n= (\mathbf{X}_n, \mathbf{Y}_n)$ be a
	random variable given by the iterative
	star-coupling with replacement of ${Z}_0, Z_1, \ldots, Z_n$.
	There exists a deterministic function
	$\Psi:\mathbf{I}_n \to \mathbf{\bar{J}}_n$ such that  $\P\big(\mathbf{\bar{Y}}_n =
	\Psi(\mathbf{X}_n)\big) > 1 - \eta$.
\end{proposition}

The proof of Proposition \ref{delJunco} uses the Shannon-McMillan-Breiman theorem  and Remark \ref{usefulb}.   A version of Proposition \ref{delJunco} is also used the proof of Theorem \ref{old} of  Quas and Soo, see \cite[Proposition 14]{Qsc}.

\section{The Proof of Proposition \ref{delJunco}}

\begin{proof}[Proof of Proposition \ref{delJunco}]
  We will place conditions on $\kinitial$ later.       Let $h:=H(\alpha) = H(\beta)$.  Let $\e>0$ such that
  \begin{equation}
  \label{choiceofe}
  h - 2\e > \big(1- \tfrac{1}{\kblock}\big)(h+ \e).
  \end{equation}   
   Set
	\[\mathbf{L}_j:= \kinitial + \kblock j,  \text{ for } 0 \leq j \leq n. \]
	%and

	%
	%
	%By Lemma \ref{mono-entropy}, for $\kmark$ sufficiently large,

	Let $\mathbf{x} \in {\mathbf{I}_j}$ be given by
	$\mathbf{x}= (x_{0}, \ldots, x_{j})$.  
%
% For every $k \in \Z^{+}$, let $\alpha^k$  and $\beta^k$ denote the $k$-fold product measures on  $A^k$ and $B^k$, respectively.   
%
    We say that $\mathbf{x}$ is
	$\alpha$-\dff{good} if
	\begin{equation}
	\label{pgood}
	\alpha^{\mathbf{L}_j}(\mathbf{x}) <
	e^{-(h - \e) \mathbf{L_j}},
	\end{equation}
	and is
	$\alpha$-\dff{completely good} if for all $ 0 \leq i \leq j$, we have
	$(x_{0}, \ldots, x_{i}) \in \mathbf{I}_i$ is good.   

The corresponding definition for $\beta$ is more complicated.  We remark  that in the presence of a strict entropy gap, $H(\alpha) > H(\beta)$,  the definition could be more simple and symmetric (see for example, \cite[page 366]{keanea} or    \cite[Proof of Proposition 14]{Qsc}.    We declare that  every  $\mathbf{y} \in \mathbf{J}_0$ is $\beta$-\dff{good}. 
	Set
\[\mathbf{\bar{L}}_j:=    (\kblock-1)j \text{ for } 0 \leq j \leq n,\]
so that $\mathbf{L}_j = \mathbf{\bar{L}}_j + \kinitial + j$.
       We say that $\mathbf{y}=(y_0,  (y_1, y_1'), \ldots,  (y_j, y_j')) \in \mathbf{J}_j$ is   $\beta$-\dff{good} if
	\begin{equation}
	\label{qgood}
	\beta^{\mathbf{{\bar{L}}}_j}(\mathbf{\bar{y}}) >
	e^{-(h -2\e) \mathbf{{L}_j}}.
	\end{equation}
    Note that being $\beta$-good  does not depend on the behavior of
    the coordinates 
     $(y_0, y_1', \ldots, y_j')$ and $\mathbf{L}_j$ appears in the exponent rather than $\mathbf{\bar{L}}_j$ on the right hand side of \eqref{qgood}.
We say that $\mathbf{y}$ is
	$\beta$-\dff{completely good} if  for all $ 0 \leq i \leq j $, we have
	$\mathbf{y}_j=(y_0,  (y_1, y_1'), \ldots,  (y_i, y_i')) \in \mathbf{J}_i $ is good.   

   Note that if $\mathbf{y} \in \mathbf{J}_n$ is not completely good, then for some $j \geq 1$, we have 
$\beta^{\mathbf{\bar{L}}_j}(\mathbf{\bar{y}}_j) <e^{-(h-2\e)\mathbf{L}_j},$
 and by \eqref{choiceofe}, 
\begin{equation}
\label{scalechange}
\beta^{\mathbf{\bar{L}}_j}(\mathbf{\bar{y}}_j) < e^{-(h-2\e)\mathbf{L}_j} \leq e^{-(1- 1/\kblock)(h+ \e) \mathbf{L}_j} \leq e^{-(h+\e) \mathbf{\bar{L}}_j}.
\end{equation}

For two elements $\mathbf{y}, \mathbf{z} \in \mathbf{J}_j$,  we say  that they are \dff{equivalent} if $\bar{\mathbf{y}} = \bar{\mathbf{z}}$.    
We let $[\mathbf{y}] \subset \mathbf{J}_i$ be the equivalence class of $\mathbf{y}$.   Given a measure on $\mathbf{I}_j \times \mathbf{J}_j$ we say it \dff{finely splits} an element $\mathbf{x} \in \mathbf{I}_j$  if there exists $\mathbf{y}, \mathbf{z} \in \mathbf{J}_i$ such that $[\mathbf{y}] \not = [\mathbf{z}]$ and for  which the measure assigns positive mass to both $(\mathbf{x}, \mathbf{y})$ and  $(\mathbf{x}, \mathbf{z})$.

	For $j \geq 0$, let  $\mathbf{W}_j=(\mathbf{X}_j, \mathbf{Y}_j)$
	be a random variable
	given by the  iterative star-coupling with replacement of $Z_{0}, Z_1,
	\ldots, Z_{j}$, where we set   $\mathbf{W}_0:={Z}_0$; thus
	$(\mathbf{X}_j, \mathbf{Y}_j)$
	takes values in  ${\mathbf{I}_j} \times \mathbf{J}_j$.
	We say that $\mathbf{x} \in  {\mathbf{I}_j}$ is
	\dff{desirable} if  the following properties are satisfied.
	\begin{enumerate}[(a)]
		\item
		The element $\mathbf{x}$ is $\alpha$-completely good.
		\item
		\label{trivb}
		The element $\mathbf{x}$ is not finely split  by (the law of)
		$\mathbf{W}_j=(\mathbf{X}_j, \mathbf{Y}_j)$.
		\item
		\label{unique}
		Furthermore,   up to equivalence, there is a unique $\beta$-completely good
		$\mathbf{y} \in  {\mathbf{J}_j}$ for which $(\mathbf{x}, \mathbf{y})$
		receives positive mass under (the law of) $\mathbf{W}_j$.
	\end{enumerate}
	For desirable $\mathbf{x} \in \mathbf{I}_j$, set $\Psi_{j}(\mathbf{x}) =
	\mathbf{\bar{y}}$,  where $\mathbf{y}$ is determined by  condition \eqref{unique};
	otherwise if $\mathbf{x}$ is not desirable  simply set
	$\Psi_{j}(\mathbf{x})= \mathbf{y}'$ for some predetermined fixed
	$\mathbf{y}'\in \mathbf{J}_j$.    Note that
	$$
	\P\big(\mathbf{\bar{Y}}_j =
	\Psi_j(\mathbf{X}_j)\big) \geq \P(\mathbf{X}_j
	\text{ is  desirable}).
	$$

	Remark \ref{usefulb} and del Junco's inductive argument \cite[Lemma 4.6]{Junco} will be used to show that for all $j \geq 0$,
	\begin{eqnarray}
	\label{pre-SMB}
	\P(\mathbf{X}_j  \text{ is not desirable})  &\leq& \P(\mathbf{X}_j
	\text{ is not c.g.}) +    \P(\mathbf{Y}_j
	\text{ is not c.g.})  +    \nonumber \\  && 
|B|^{\kblock}\sum_{i=0} ^ {j-1} e^{-\e	\mathbf{L}_i},
	\end{eqnarray}
	where ``c.g.'' is short for completely good.
	\begin{comment}
	\begin{equation}
	\label{max}
	M:= %\max_{0 \leq i \leq j}
	\# (\nset{N}^{\kmark
	}).
	\end{equation}
	\end{comment}
	
	The case $j=0$ is vacuous, since being good implies being completely good,
	and under  ${Z}_0$ no elements  are finely split.

	Assume \eqref{pre-SMB} for the case $j-1 \geq 0$.   We show that
	\eqref{pre-SMB} holds for the case $j$.  Let  $E$ be the event that
	$\mathbf{X}_{j-1}  \text{ is desirable, but }
	\mathbf{X}_{j} \text{ is not desirable}$.   Clearly,
	\begin{equation}
	\label{cases}
	\P(\mathbf{X}_j  \text{ is not desirable}) \leq  \P(\mathbf{X}_{j-1}
	\text{ is not desirable}) +
	\P(E).
	\end{equation}
	Note that on the event $E$,
	the random variables $\mathbf{X}_{j-1}$ and $\mathbf{Y}_{j-1}$ are completely good.
	Observe that the event $E$ is contained in the following three events
	\begin{enumerate}[(I)]
		\item
		$E_1:=$  The random variable $\mathbf{X}_j$ is not good, but $\mathbf{X}_{j-1}$
		is completely good.
		\item
		$E_2:=$ The random variable $\mathbf{X_j}$ is completely good, but is finely split
	 under the iterative star-coupling $\mathbf{W}_j$,
		even though $\mathbf{X}_{j-1}$ is desirable.
		\item
		%$E_3:=$   The random variable $\mathbf{X}_{j}$ is completely good and  is not split by the iterative star-coupling $\mathbf{W}_J$, but $\mathbf{Y_j}$ is not good, even though $\mathbf{Y}_{j-1}$ is completely good.
		$E_3:=$  The random variable  $\mathbf{Y_j}$ is not good, but
		$\mathbf{Y}_{j-1}$ is completely good.
	\end{enumerate}
	Clearly,
	\begin{equation}
	\label{casea}
	\P(E_1) + \P(\mathbf{X}_{j-1} \text{ is not c.g.}) =
	\P(\mathbf{X}_j \text{ is not c.g.}).
	\end{equation}
	Similarly,
	\begin{equation}
	\label{casec}
	\P(E_3) + \P(\mathbf{Y}_{j-1} \text{ is not c.g.}) =
	\P(\mathbf{Y}_j \text{ is not c.g.}).
	\end{equation}

	Let us focus on the event $E_2$.  Let $\mathbf{X}_{j} = ( \mathbf{X}_{j-1},X)$,
	so that $X$ takes values in $A^{\kblock}$.   We show that for any
	$x \in A^{\kblock}$ and any completely good $\mathbf{y} \in
	{\mathbf{J}_{j-1}}$ that
	\begin{equation}
	\label{use-usefulb}
	\P(E_2 \ | \ X=x, \mathbf{\bar{Y}}_{j-1} = \mathbf{\bar{y}}) \leq  |B|^{\kblock} e^{-\e \mathbf{L}_{j-1}},
	\end{equation}
	so that $\P(E_2) \leq |B|^{\kblock} e^{-\e\mathbf{L}_{j-1}}$ and it follows that
	\eqref{pre-SMB} holds by \eqref{cases}, \eqref{casea}, \eqref{casec}, and
	the inductive hypothesis.
	
	Note that if $\mathbf{x}$ and $\mathbf{y}$ are good, then
	\begin{eqnarray}
	\P( \mathbf{X}_{j-1} = \mathbf{x} |X=x, \mathbf{\bar{Y}}_{j-1} = \mathbf{\bar{y}}) &=&
	\frac{\P(\mathbf{X}_{j-1} = \mathbf{x},  \mathbf{\bar{Y}}_{j-1} =
		\mathbf{\bar{y}}, X=x)}{\P(\mathbf{\bar{Y}}_{j-1} = \mathbf{\bar{y}}, X=x)} \nonumber \\
	\label{first}
	&=& \frac{\P(\mathbf{X}_{j-1} = \mathbf{x},  \mathbf{\bar{Y}}_{j-1} = \mathbf{\bar{y}})}
	{\P(\mathbf{Y}_{j-1} = \mathbf{\bar{y}})}  \\
	%\label{second}
	&\leq&  \frac{\P(\mathbf{X}_{j-1} = \mathbf{x})}{\P(\mathbf{\bar{Y}}_{j-1} =
		\mathbf{\bar{y}})} \nonumber\\
	\label{division}
	& \leq & e^{-\e \mathbf{L}_{j-1}},
	\end{eqnarray}
	where  \eqref{first} follows from the independence properties of the star-coupling (with replacement) and \eqref{division}
	follows from \eqref{pgood} and \eqref{qgood}.
	Also note that if $\mathbf{x}$ is desirable, then if  $(\mathbf{x}, x)$ is finely split under $\mathbf{W}_j$, then for 
	the unique, up to equivalence,  $\mathbf{y}$ for which
	$(\mathbf{x}, \mathbf{y})$ receives positive mass under $\mathbf{W}_{j-1}$
	there exist $(y,u), (y', u') \in B^{\kblock} = B^{\kblock-1} \times B$ such that $y \not = y'$ and 
	both $((\mathbf{x}, x), (\mathbf{y}, y, u))$ and
	$((\mathbf{x}, x), (\mathbf{y}, y', u'))$
	receive positive mass under $\mathbf{W}_j$.
	By Remark \ref{usefulb} and the definition of the star-coupling with replacement, 
	for a fixed $x \in A^{\kblock}$ and $\mathbf{y} \in {\mathbf{J}_{j-1}}$
	the set of all $\mathbf{x}$ such that there exists distinct  $y,y'\in B^{\kblock -1}$
	for which there are $u, u' \in B$ such that both $((\mathbf{x}, x), (\mathbf{y}, y,u))$ and $((\mathbf{x}, x), (\mathbf{y}, y',u'))$
	receive positive mass under $\mathbf{W}_j$ has at
	most $ |B|^{\kblock-1}-1$ elements; thus  summing over all such $\mathbf{x}$
	yields \eqref{use-usefulb}.

	%Using the SMB theorem,
	%with Lemma \ref{mono-entropy}
	%the argument
	%given for the proof of \cite[Proposition 4.7]{Junco} implies that $\kinitial$
	The Shannon--McMillan--Breiman theorem implies that $\kinitial$
	can be chosen so that all
	three  terms in \eqref{pre-SMB} can be made smaller than $\eta/3$.
	This is done in the following way.
	Set
\begin{equation*}
		S_{A}(k, K):=  
		\ns{\mathbf{a} \in A^K:  \alpha^{\ell}(\mathbf{a})
			< e^{-(h-\e)\ell} \text{ for all }  k \leq \ell \leq K} \\
\end{equation*}	
and
\begin{equation*}
		S_{B}(k, K):=
		 \ns{\mathbf{b} \in B^K:  \beta^{\ell}(\mathbf{b}) >
			e^{-(h + \e)\ell} \text{ for all }  k \leq \ell \leq K},
\end{equation*}
where  we have the slight abuse of notation that  
if $\mathbf{a} = (a_1, \ldots, a_K)$, then  
$\alpha^{\ell}(\mathbf{a}) = \alpha^{\ell}(a_1, \ldots, a_{\ell})$.

	First, by the Shannon--McMillan--Breiman theorem choose $\kappa$ so that for all
	$K \geq \kappa$, we have
\begin{equation}
\label{SMB-bom}
\beta^K(S_{B}(\kappa, K)) > 1-\eta/3.
\end{equation}
Next, using the Shannon--McMillian--Brieman theorem again, choose $\kinitial$ sufficiently large so that  the following three inequalities are satisfied:
\begin{equation}
\label{SMB-top}
\alpha^K(S_{A}(\kinitial, K)) > 1-\eta/3 \text{ for all $K \geq \kinitial$}, 
\end{equation}
\begin{equation}
\label{min}
\min\ns{\beta^{\ell}(\mathbf{y})>0:  \mathbf{y} \in B^{\ell}, 0 \leq \ell \leq \kappa} > e^{-(h - 2\e)\kinitial},
\end{equation}
and
	\begin{equation}
	\label{pile}
	|B|^{\kblock }\sum_{i=\kinitial}^{\infty} e^{-\e i} < \eta/3.
	\end{equation}
%	
	\begin{comment}
	\P(\mathbf{X}_j
	\text{ is not c.g.}) +    \P(\mathbf{Y}_j
	\text{ is not c.g.}) + {N}^{\kmark
	}\sum_{i=0} ^ j e^{-\delta	\mathbf{L}_j},
	\end{comment}
%	

Finally, we will verify that this choice of $\kinitial$ is sufficient.    Condition \eqref{SMB-top}  gives that $\P(\mathbf{X}_j
	\text{ is not c.g.}) < \eta/3$.   Recall that   by definition, $\mathbf{L}_j\geq \kinitial$ for all $j \geq 0$, so that    \eqref{pile} ensures that
	 $|B|^{\kblock}\sum_{i=0} ^ {j-1}
	e^{-\e	\mathbf{L}_j} < \eta/3$.   
It remains to verify that $\P(\mathbf{Y}_j \text{ is not c.g.}) <
	\eta/3$.
  
The  definition of completely good, \eqref{qgood},  gives that if $$\mathbf{y}=(y_0,  (y_1, y_1'), \ldots,  (y_j, y_j')) \in \mathbf{J}_n$$ is not completely good, then for some $i > 0$, we have      
\begin{equation}
\label{implies}
\beta^{\mathbf{\bar{L}}_i}({\mathbf{\bar{y}}_i}) < e^{-(h-2\e)\mathbf{L}_i} < e^{-(h-2\e)\kinitial},
\end{equation}
where $\mathbf{y}_i = (y_0,  (y_1, y_1'), \ldots,  (y_i, y_i'))$; inequalities \eqref{implies} and \eqref{min} imply that 
\begin{equation}
\label{kappa}
\mathbf{\bar{L}}_i > \kappa;
\end{equation}
   moreover, \eqref{scalechange} gives  that
\begin{equation}
\label{returnscale}
 \beta^{\mathbf{\bar{L}}_i} (\mathbf{\bar{y}}_i)   \leq e^{-(h+\e) \mathbf{\bar{L}}_i}.
\end{equation}
Hence if $\mathbf{y}$ is not completely good, then by \eqref{kappa} and \eqref{returnscale} it belongs to the complement of $S_B(\kappa, K)$ for all $K \geq \kappa$.  Thus \eqref{SMB-bom} gives that $\P(\mathbf{Y}_j \text{ is not c.g.}) <
	\eta/3$.  
\end{proof}

In Proposition \ref{delJunco}, we have that given $\mathbf{x} \in \mathbf{I}_n$, with high probability, up to equivalence, it determines a corresponding  $$\mathbf{y} =(y_0,  (y_1, y_1'), \ldots,  (y_n, y_j')) \in \mathbf{J}_n.$$  It will be useful to refer to $(y_0, y_1', \ldots, y_n')$ as the \dff{undetermined coordinates}, and $(y_1, \ldots, y_n)$ as the \dff{destined  coordinates}.   We say that there are $\kinitial + n$ undetermined coordinates, since $y_0 \in B^{\kinitial}$, and $y_i \in B$ for $1 \leq i \leq n$,  and there are $(\kblock -1)n$ destined coordinates.

\section{Perturbing the joining}

Let $\xi \in J$ be a monotone joining of marker form.   We will define a perturbation $\xi'$ of $\xi$ using the iterated star-coupling with replacement.  The perturbation will depend on a few parameters.   With the help of Proposition \ref{delJunco}, we will be able to make a choice of these parameters so that $\xi'$ will be an almost factor and close to $\xi$ in the metric defined in \eqref{metric}.

\subsection{Defining the perturbation}
Let $\kmark  < \krmark$ be  large integers to be chosen later.   
 A \dff{secondary marker} is the maximal union of at least $\kmark$   consecutive primary markers, so that secondary markers have no filler between them and if the interval $[i, j]$ is a secondary marker for $x \in \X$, then $x$  restricted to $[i,j]$ has the form $0101 \cdots 0101$.

   Similarly, a  \dff{tertiary marker} is the maximal union of at least $\krmark$ consecutive primary markers.     We call the set of integers between but not including two secondary markers  a \dff{block}, and the set of integers between but not including two tertiary markers a \dff{city}.  Thus within a city there are blocks, which we consider ordered from left to right.  Note that a block may contain primary markers.  

Let $p \in (\tfrac{1}{2}, 1)$.  Let $\xi \in J(p)$.   Let $Z=(X, Y)$ have law $\xi$.   Suppose that we are given that  $Z$ has primary  markers given by  $t \in \T$.      Let $I \subset \Z$ be a  block of  length $n$.          
We are interested in the distribution of the random  variable  taking values in $\ns{0,1}^n \times \ns{0,1}^n$ given by  
the distribution of $Z$, conditioned on $t$, restricted to $I$.  
The \dff{type} of the block  $I$ is defined to be the vector containing an alternating sequence of integers that are the lengths of the filler and marker intervals in $I$ and the \dff{length} of the type is simply the sum of the integers in the type which give the length of the block.      The distribution of this random variable is determined by the parameter $p$ and the type of $I$.   There are a countable number of types.  Fix an enumeration $(\type_i)_{i \in \N}$ of the types and let $\rho_i$ be the corresponding law.   
Associate to each type-$i$ block  a large integer $\kinitial^i$  which will be chosen later; here $i$ is an index that is not an exponent.          A \dff{census} of a city is the sequence of nonnegative integers $c_i$, where each $c_i$ is the number of type-$i$ blocks in the city.   

\begin{remark}
\label{equalentropy}
Note that for every $i \in \N$, if the length of the type-$i$ is $n$, then $\rho_i$ is a probability measure on $\ns{0,1}^n \times \ns{0,1}^n$ with projections $\alpha_i$ and $\beta_i$ that have equal entropy.  The equal entropy assertion also follows from the duality between $p$ and $1-p$.  \erk
\end{remark}

A \dff{modification} of $Z =(X,Y)$ on a subset of $\Z$ is a coupling of $X$ and $Y$ given by  $Z'=(X',Y')$ such that $Z'$ is equal to $Z$ off the subset and has the same primary markers as $Z$.    We will define a modification $Z'$ of $Z$ so that the law of $Z'$ will be a member of $J$.     The modifications will be made independently on each city, so that we need only define what changes occur on a  city.       On each city the modifications will be made independently on each set of types, so that we need only define what changes occur on each set of types.      

\begin{comment}
We will make two types of modifications.   After the first round of modifications, we will end up with a random variable $Z_{\mathrm{prod}}$.   We call this the \dff{product modification} of $Z$.     Suppose that the primary markers of $Z$ are given by $t \in \T$.  Let us focus on a the type-$i$ blocks in a city. 
Suppose that length of a type-$i$ block is $n$.  It may be helpful to think of two different copies of $\ns{0,1}^n$ by setting $A=B = \ns{0,1}^n$.     Let $W=(W_j)_{j=1} ^{c_i}$ be  the set of random variables taking values in $A \times B$ obtained by taking the restriction of $Z$, conditioned to have primary markers given by  $t$,  to each block of type-$i$ in the city.   Although $W$ gives an identical sequence, where each $W_j$ has law $\rho_j$,  it may not be independent.  However,  the projections on $A$ and $B$ are independent; if we write $W_j=(X_j, Y_j)$, then by Remark \ref{iid}, we have that  $X=(X_j)$ and  $Y=(Y_j)$ are i.i.d.\ sequences.   We define a modification of $Z$ be replacing the values of $W$ with those of $W'$, there $W'$ is an i.i.d.\ sequence of random variables all with law $\rho_i$ (that is also independent of $Z$).    For a single type-$i$, if we apply this modification on each city, independently, then we obtain a modification of $Z$.  Repeat, this modification for all the types to obtain $Z_{\mathrm{prod}}$.  
\end{comment}

  Suppose that the primary markers of $Z$ are given by $t \in \T$. Fix $i \in \N$.   Let us focus on the type-$i$ blocks in a single fixed city.    We will refer to this modification as the \dff{star-modification of type-$i$} on a city.    
Suppose that the census $c$ is such that we may write
\begin{equation}
\label{euclid}
c_i =  \kinitial^{i} + q_i\kblock + r_i, 
\end{equation}
where $0 \leq r_i < \kblock$ and $q_i$ is an nonnegative integer.   We will not make modifications on the last $r_i$ blocks.  
Suppose that length of the type-$i$ block is $n$. 
 It may be helpful to think of two different copies of $\ns{0,1}^n$ by setting $A=B = \ns{0,1}^n$.     Let $W=(W_j)_{j=1} ^{c_i}$ be  the set of random variables taking values in $A \times B$ obtained by taking the restriction of $Z$, conditioned to have primary markers given by  $t$,  to each block of type-$i$ in the city.   Although $W$ gives an identical sequence, where each $W_j$ has law $\rho_j$,  it may not be independent.  However,  the projections on $A$ and $B$ are independent; if we write $W_j=(X_j, Y_j)$, then by Remark \ref{iid}, we have that  $X=(X_j)$ and  $Y=(Y_j)$ are i.i.d.\ sequences.         Consider the first $\kinitial^i$ random variables together as a single random variable taking values in $(A\times B)^{\kinitial^i}$, and each subsequent $\kblock$ random variables together as  random variables taking values in $(A \times B)^{\kblock}$.  We obtain a sequence of random variables $M=(M_0, M_1, \ldots, M_{q_i})$.   Thus $M$ takes values on $$(A \times B)^{\kinitial ^i+ q_i\kblock} \equiv A^{\kinitial^i + q_i\kblock} \times B^{\kinitial ^i + q_i\kblock}.$$  Take the iterative star-coupling with replacement of these random variables to obtain new random variables $M'=(M_0', \ldots, M_{q_i}')$; furthermore, using independent randomization, we may stipulate that these random variables are independent of $Z$.       We define a modification $Z'$ of $Z$ by replacing the values of $M$ with those of $M'$, so that $Z=Z'$ off the type-$i$ blocks in the city, and if $Z$ restricted to the type-$i$ blocks,  then it is given by $M$, then $Z'$ restricted to the type-$i$ blocks is given by $M'$.    The iterated star-coupling with replacement gives that the law of each $M_j'$ projected onto each of the $\kinitial +q_i\kblock$ copies of $A \times B$ is $\rho_i$, so that monotonicity is preserved and the primary markers remain unchanged.  Also, the projections of $M$ and $M'$ on $A^{c_i - r_i}$ have the same law.  Similarly,  the projections of $M$ and $M'$ on $B^{c_i - r_i}$ have the same law, so that by Remark \ref{iid}, the random variable  $Z'$ gives the required coupling.

Note we have only defined the star modification of type-$i$  when $c_i \geq \kinitial ^i + \kblock$.     In the case that $c_i$ is not sufficiently large, we simply do nothing, that is, we stipulate that the  star modification of type-$i$ leaves everything unchanged.

For a single type-$i$, if we apply the star modification of the type on each city, independently, then we obtain a modification of $Z$ that has law that belonging to $J$.   We call this the \dff{star-modification of type-$i$} of $Z$.      We summarize our construction in the following proposition. 

\begin{proposition}
	\label{sum}
Let $p \in (\tfrac{1}{2}, 1)$ and  $\xi \in J(p)$.     If $Z$ has law $\xi$ and if $Z'$ is a star modification of $Z$ of a particular type,   then the law of $Z$' is also a  member of $J(p)$. 
\end{proposition}

Given a finite set of types, the \dff{star-modification} of $Z$ on the set of types is obtained by applying the star-modification in succession, starting with the smallest type.

\begin{remark}
	\label{dblock}
	In our construction of the star modification of type-$i$ on a city, we relied on the fact that the law of $M_j'$ still has a projection of $\rho_i$ on each copy of $A \times B$ to ensure that primary markers and monotonicity are preserved.  This fact will also  be important for us later in proving that the parameters of the star-modification $Z'$ of $Z$ can be chosen so that it is a small perturbation in  the $\metric$-metric, since on the event that the origin is contained in a block, and  the coordinates of a cylinder set  $C$ lie in that block, we have that the probabilities of $C$ under $Z$ and $Z'$ are not only close, they are  equal!  This is another one of nice features of del Junco's star coupling.     \erk
\end{remark}

\subsection{Choosing the parameters}

From the discussion in Section \ref{baire}, it remains to show that given a joining $\xi \in J(p)$, we can choose parameters so that the star-modification of a random variable with law $Z$ on a finite set of type results in a random variable with law $\xi'$ that is close to $\xi$ in the metric defined by \eqref{metric} and is also an almost factor.

\begin{proof}[Proof of Theorem \ref{mresult}]
Let $p \in (\tfrac{1}{2},1)$.  Let $\e >0$.    As discussed in Section \ref{baire}, it suffices to show that $U_{\e}$, the set of $\e$-almost factors from $B(1-p, p)$ to $B(p, 1-p)$ is dense.   The proof that $V_{\e}$,   the set of $\e$-almost factors from $B(p, 1-p)$ to $B(1-p, p)$ is similar  with the roles of $p$ and $1-p$ reversed. 

Let $\xi \in J(p)$ and $Z$ have law $\xi$.   Let $\e>0$.  We will choose the parameters for the star-modification $Z'$ of $Z$ as follows.  The modification will occur on a finite set of types $\mathcal{T}$, which will be specified later.   Recall that in the star-modification, some blocks are left \dff{unchanged}, so that $Z$ equals $Z'$ on those blocks, and  whereas some blocks are \dff{modified} via the iterated star-coupling with replacement, so that $Z$ may not equal to $Z'$ on those blocks.   Note that $Z'$ and $Z$ always share the same primary markers, and although markers may lie in the  modified coordinates they are always preserved.       If $\xi'$ is the law of $Z'$, then these parameters will be chosen so that $\metric(\xi, \xi') < \e$ and $\xi' \in U_{\e}$.  
We choose the parameters as follows.
\begin{enumerate}[(i)]
\item
\label{epscale}
Set $\e'  := \e/100$.
\item
\label{deltaweakstar}
Let $\delta >0$ be small enough and $\ellstar$ be large enough so that two measures $\zeta$ and $\zeta'$ on $\ns{0,1}^{\Z} \times \ns{0,1}^{\Z}$ are $\e'$ close in the  metric $\metricstar$, if for all cylinder sets  $C \in \mathcal{C}_{\ellstar}$, we have 
$|\zeta(C) - \zeta'(C)| < \delta.$
\item
\label{kmarkone}
Choose $\kmark$ sufficiently large so that with probability at least $1 - \e'$ the origin is in a block and the interval $[-2\ellstar, 2\ellstar]$ is in the block.
\item
\label{kmarktwo}
With this choice of $\kmark$, there exists $L >0$ such that with probability at least $1- 2\e'$ the origin will in a block and the length of the block will be between $\ell^*$ and $L$.    
\item
\label{kmarkthree}
 In particular, there exists a finite set of types $\mathcal{T}$,  those with lengths between $\ellstar$ and $L$,  such that with probability at least $1- 2\e'$, each block will be of type $\mathcal{T}$.   Since we have a fixed enumeration of the types, we will view $\mathcal{T}$ as a subset of $\N$.  
\item
\label{group}
Set $\kblock = \lceil 1/\e' \rceil +1.$  
\item
\label{eta}
 For each $i \in \N$, choose $\kinitial^i$ via Proposition \ref{delJunco}, by substituting $\rho = \rho_i$, $\kinitial=\kinitial^i$, and $\eta = \e'$.  
\item
\label{quota}
Let $c$ be the census of the city containing the origin.   If $c_i$ is sufficiently large, define  $q_i$  as in \eqref{euclid}.  Choose  $\krmark$ sufficiently large   so that with probability at least $1- \e'$, the origin is in a city, and  for all $i \in \mathcal{T}$ the census will satisfy  $\kinitial^i / q_i < \e'$. 
\end{enumerate}

  Applying Proposition \ref{sum} a finite number of times gives that  $\zeta' \in J$.   By Remark \ref{dblock}, conditions \eqref{epscale}, \eqref{deltaweakstar}, and  \eqref{kmarkone}, imply that $\metric(\zeta', \zeta) < \e$.   
 
 It remains to verify that $\zeta' \in U_{\e}$.  Call $t \in \T$ a \dff{model marker} if the block containing the origin is a modified block and 
the origin lies in a destined coordinate.   Property \eqref{eta} and Proposition \ref{delJunco} imply that for all model markers $t \in \T$ there exists a deterministic measurable  $\psi : \Omega \to \ns{0,1}$ such that 
\begin{equation}
	\label{psi}
\zeta_t'\ns{(x,y): (x, y_0) = (x, \psi(x)   )} >  1- \e'.
\end{equation}
\begin{comment}
We note that on the event that the zeroth coordinate lies in a primary marker, then we immediately can deduce the value of $y_0$ from $x$, since $x$ and $y$ have the same primary markers.
 \end{comment}
For a particular type-$i$, with  $c_i = \kinitial ^i + q_i \kblock + r_i$ as in \eqref{euclid}  the ratio of undetermined coordinates plus those that are unchanged to destined coordinates is $(\kinitial^i +  q_i  + r_i) /q_i (\kblock-1)$.  Recall that $r_i < \kblock$.    Conditions \eqref{kmarktwo}, \eqref{kmarkthree}, \eqref{group}, and  \eqref{quota}, ensure us that the set of model markers has  probability at least $1- 7\e'$; this fact together with \eqref{psi} and \eqref{epscale} imply that $\zeta' \in J$.   
\end{proof}

\section{Some other examples}

One of  the key observations of Keane and Smorodinsky \cite[Lemmas 2 and 3]{keaneb} that allowed the definition of markers in their proof of that two Bernoulli shifts $B(\mathbf{p})$ and $B(\mathbf{q})$ of equal entropy are isomorphic was that one could assume without loss of generality that ${p}_0 = {q}_0$ in the case  where $\mathbf{p}$ and  $\mathbf{q}$ give non-zero mass to  three or more symbols, and in the case where $\mathbf{p}$ gives non-zero mass to only two symbols, then one can assume that ${p}_0^k{p}_1 = {q}_0^k{q}_1$ from some $k$. 
 In general, in the construction of monotone factors, we may not make this reduction since monotonicity may not preserved.  However by a straightforward adaptation of the proof of Theorem \ref{mresult}, the following monotone versions of the Keane and Smorodinsky reductions are enough to prove the existence of a monotone isomorphism.   

\begin{theorem}
	\label{onemark} 
	Let $N \geq 2$.  Let $\mathbf{p}$ and $\mathbf{q}$ be probability measures on $[N]$ of equal entropy.  Suppose $\mathbf{p}$ stochastically dominates  $\mathbf{q}$, and  furthermore there exists $i \geq j$ such that ${p}_i = {q}_j$ and $\mathbf{p}^*$ stochastically dominates $\mathbf{q}^*$, where $\mathbf{p}^*$ is the law of a random variable with law $\mathbf{p}$ conditioned not to take the value $i$, and $\mathbf{q}^{*}$ is the law of a random variable with law $\mathbf{q}$ conditioned not to take the value $j$.  Then there exists a monotone isomorphism of $B(\mathbf{p})$ and  $B(\mathbf{q})$.  
	\end{theorem}

\begin{theorem}
\label{twomark}
	Let $N \geq 2$.  Let $\mathbf{p}$ and $\mathbf{q}$ be probability measures on $[N]$ of equal entropy.  Suppose $\mathbf{p}$ stochastically dominates  $\mathbf{q}$, and  furthermore there exists $i \geq j$ and $k \geq \ell$ such that ${p}_i {p}_k = {q}_j {q}_{\ell}$ and for all $n \geq 1$, we have that  $\mathbf{p}^{n*}$ stochastically dominates $\mathbf{q}^{n*}$,  where   $\mathbf{p}^{n*}$ is the law of a random vector with law $\mathbf{p}^n$ conditioned so that an occurrence of an $i$ is never immediately followed by an occurrence of a $k$, and $\mathbf{q}^{n*}$ is the law of a random vector with law $\mathbf{q^n}$ conditioned so that an occurrence of a $j$ is never followed by an occurrence of an  $\ell$.    Then there exists a monotone isomorphism of $B(\mathbf{p})$ and  $B(\mathbf{q})$.
\end{theorem}

\begin{comment}
Let $\mathbf{p} = (\tfrac{4}{12}, \tfrac{3}{12}, \tfrac{5}{12})$ and $\mathbf{q} = (\tfrac{5}{12}, \tfrac{4}{12}, \tfrac{3}{12})$ be probability measures on $\nset{2} = \ns{0,1,2}$.  One can easily verify that $\mathbf{p}$ stochastically dominates $\mathbf{q}$.  However, Theorem \ref{onemark} does not apply.  
\end{comment}

\section*{Acknowledgement}

I thank Zemer Kosloff for his  help with Example \ref{mesh}.  I also thank the referee for the careful reading of this paper and useful suggestions.

\newcommand{\SortNoop}[1]{}

%\bibliographystyle{abbrv}
%\bibliography{embedding}

\end{document}